\documentclass[12pt,reqno]{amsart}
\usepackage{amssymb,hyperref}
\usepackage[all]{xy}

\newcommand{\op}{\operatorname}

\newtheorem{theorem}{Theorem}[section]
\newtheorem{lemma}[theorem]{Lemma}
\newtheorem{proposition}[theorem]{Proposition}
\newtheorem{corollary}[theorem]{Corollary}

\theoremstyle{definition}
\newtheorem{definition}[theorem]{Definition}
\newtheorem{remark}[theorem]{Remark}

\numberwithin{equation}{section}

\usepackage{lipsum}

\usepackage[symbol]{footmisc}

\begin{document}

\title[Fundamental groups of moduli of principal bundles]{Fundamental groups of 
moduli of principal bundles on curves}

\author[I. Biswas]{Indranil Biswas}

\address{School of Mathematics, Tata Institute of Fundamental
Research, Homi Bhabha Road, Mumbai 400005, India}

\email{indranil@math.tifr.res.in}

\author[S. Mukhopadhyay]{Swarnava Mukhopadhyay}

\address{School of Mathematics, Tata Institute of Fundamental
Research, Homi Bhabha Road, Mumbai 400005, India}
\thanks{Corresponding author: Swarnava Mukhopadhyay}

\email{swarnavaster@gmail.com, swarnava@math.tifr.res.in}

\author[A. Paul]{Arjun Paul}

\address{International Centre for Theoretical Sciences, Survey No. 151, Shivakote, Hesaraghatta Hobli, Bengaluru - 560 089, India}

\email{arjun.math.tifr@gmail.com}



\begin{abstract}
Let $X$ be a compact connected Riemann surface of genus $g$, with $g\, \geq\, 2$, and let $\op{G}$ 
be a connected semisimple affine algebraic group defined over $\mathbb C$. Given any 
$\delta \, \in \, \pi_1(\op{G})$, we prove that the moduli space of semistable principal 
$\op{G}$--bundles over $X$ of topological type $\delta$ is simply connected. More generally,
if $\op{G}$ is a connected reductive complex affine algebraic group, then the fundamental group of the moduli
space is isomorphic to ${\mathbb Z}^{2gd}$, where $d$ is the complex dimension of the center of $\op{G}$.
In contrast, the fundamental group of the moduli stack of principal $\op{G}$--bundles 
over $X$ of topological type $\delta$ is shown to be isomorphic to $H^1(X,\, 
\pi_1(\op{G}))$, when $\op{G}$ is semisimple. We also compute the fundamental group of the moduli stack of principal
$\op{G}$--bundles when $\op{G}$ is reductive.
\end{abstract}

\maketitle

\section{Introduction}

Let $X$ be an irreducible smooth complex projective curve, or, equivalently, a compact
connected Riemann surface. Let $\op{G}$ be a connected reductive affine algebraic group
defined over $\mathbb C$. The topological types of holomorphic principal $\op{G}$--bundles
over $X$ are parametrized by $\pi_1(\op{G})$ (see \cite[p.~186, Proposition 1.3(a)]{BLS:94},
\cite[Section 5]{Ho}). For any $\delta\, \in\, \pi_1(\op{G})$,
let ${\rm M}^{\delta}_{\op{G}}$ denote the moduli space of semistable principal
$\op{G}$--bundles over $X$ of topological type $\delta$. These moduli spaces have been 
extensively studied for the last twenty years. Our aim here is to compute the
fundamental group of ${\rm M}^{\delta}_{\op{G}}$.

When $\text{genus}(X)\, =\, 0$, then ${\rm M}^{\delta}_{\op{G}}$ is a point; this 
follows from the facts that any holomorphic principal $\op{G}$--bundle over ${\mathbb 
C}{\mathbb P}^1$ admits a reduction of structure group to a maximal torus of $\op{G}$ 
\cite[p.~122, Th\'eor\`eme~1.1]{Gr}, and the holomorphic line bundles on ${\mathbb 
C}{\mathbb P}^1$ are classified by their degree. When $\text{genus}(X)\, =\, 1$, 
there are explicit descriptions of ${\rm M}^{\delta}_{\op{G}}$ \cite{FMW}, \cite{FM}, \cite{La}. So 
we assume that $g\, :=\, \text{genus}(X)\, >\, 1$.

There is a short exact sequence of groups
$$
1\, \longrightarrow\, [\op{G},\, \op{G}]\, \longrightarrow\,\op{G}
\, \stackrel{q}{\longrightarrow}\, \op{Q}\,:=\, \op{G}/[\op{G},\, \op{G}]
\,\cong\, ({\mathbb G}_m)^d\, \longrightarrow\, 1\, ,
$$
where $d$ is the dimension of the center of $\op{G}$. Let
$$
J^{\alpha}_{\op{Q}}(X)\, \cong\, \text{Pic}^0(X)^d
$$
be the moduli space of all holomorphic principal $\op{Q}$--bundles on $X$ of
topological type $\alpha\,=\, q_*(\delta)$. The above homomorphism $q$ induces
a morphism of moduli spaces
$$
\widetilde{q}\, :\, {\rm M}^{\delta}_{\op{G}}\,\longrightarrow\,
J^{\alpha}_{\op{Q}}(X)
$$
which is in fact an \'etale locally trivial fibration (see
the proof of Corollary \ref{cor4}). We prove the following
(see Corollary \ref{cor4}).

\begin{theorem}\label{thm0}
The homomorphism of fundamental groups
$$
\widetilde{q}_*\, :\, \pi_1({\rm M}^{\delta}_{\op{G}})\,\longrightarrow\,
\pi_1(J^{\alpha}_{\op{Q}}(X))\,\cong\, {\mathbb Z}^{2gd}
$$
induced by the above projection $\widetilde{q}$ is an isomorphism.
\end{theorem}

Theorem \ref{thm0} actually extends to the more general case of any connected complex
affine algebraic group (see Remark \ref{remf}). As a consequence of Theorem \ref{thm0} we know that if $\op{G}$ is semisimple, then ${\rm 
	M}^{\delta}_{\op{G}}$ is simply connected. 

We note that Theorem \ref{thm0} was proved earlier in \cite{BLR} under the
assumption that $\delta\, =\, 1$. The method of \cite{BLR} does not extend when
$\delta$ is nontrivial; the crucial Lemma 2.4 in \cite{BLR} fails to extend (also
Corollary 2.2 in \cite{BLR} does not extend).

The proof of Theorem \ref{thm0} uses uniformization theorems \cite{BLS:94, DS, KNR}, for moduli stack of bundles 
and unirationality of $\op{M_G^{\delta}}$ for a semi-simple group $\op{G}$. For example, if we take 
$\op{G}\,=\,\op{SL}(r)$ and $\delta$ be an integer coprime to $r$, then it is well known that the corresponding moduli space is a 
projective, smooth, unirational Fano variety \cite{BL, BLS:94, KNR} and hence simply connected \cite{Serre, Kollar}. However the varieties 
$\op{M_G^{\delta}}$ for general $\op{G}$ and $\delta$ are not always smooth and hence we need use different methods 
to address these issues.

We first consider, the fundamental group of the moduli stack $\mathcal{M}^{\delta}_{\op{G}}$ (see Sections 
\ref{sec:fundtwistedstack} for a definition) of principal $\op{G}$--bundles over $X$ of topological type $\delta$. 
We prove the following (see Theorem \ref{fundamentalgroup}):

\begin{theorem}The fundamental group $\pi_1(\mathcal{M}^{\delta}_{\op{G}})$ 
is isomorphic to $H^1(X,\, \pi_1(\op{G}))$ when $\op{G}$ is semisimple.
\end{theorem}
 
It should be mentioned that more generally, when $\op{G}$ is reductive, the fundamental group of
the moduli stack of principal $\op{G}$--bundles over $X$ is computed in Corollary
\ref{fundamentalreductivestack}. As an example if we take $\op{G=PGL}(r)$, then for any $\delta$, the fundamental group of the moduli stack is $(\mathbb{Z}/r\mathbb{Z})^{2g}$, where as the corresponding moduli space is simple connected. 

To give a rough reason why $\pi_1({\rm M}^{\delta}_{\op{G}})$ vanishes for
$\op{G}$ semisimple, first consider 
the action of the group $H^1(X,\, \pi_1(\op{G}))$ on any twisted moduli space (see Sections \ref{sec:fundtwistedstack} and \ref{sec:twistMS} for definitions) of 
semistable principal $\widetilde{\op{G}}$--bundles on $X$, where $\widetilde{\op{G}}$ 
denotes the universal cover of $\op{G}$. This action has the property that the
subgroup of $H^1(X,\, \pi_1(\op{G}))$ generated by all the isotropy subgroups is 
$H^1(X,\, \pi_1(\op{G}))$ itself. As a consequence of a general result of \cite{Am},
this makes the corresponding quotient by $H^1(X,\, \pi_1(\op{G}))$,
of the twisted moduli space under consideration,
a simply connected space, because the twisted moduli space is simply connected. Finally,
the quotient by $H^1(X,\, \pi_1(\op{G}))$ of a twisted moduli space of
semistable principal $\widetilde{\op{G}}$--bundles
is isomorphic to the moduli space ${\rm M}^{\delta}_{\op{G}}$, where $\delta \,\in \,
\pi_1(\op{G})$ is the element used in the construction of the twisted moduli space under
consideration.

We now give an application of Theorem \ref{thm0}. If $Y$ is a proper variety over an
algebraically closed field, there is an isomorphism
$$
\op{Hom}(\pi_{1}^{{\acute{e}t}}(Y),\, \mathbb{Z}/n\mathbb{Z})\,
\stackrel{\cong}{\longrightarrow}\,H^1_{{\acute{e}t}}(Y,\,\mathbb{Z}/n\mathbb{Z})
$$
for any $n$. From the long exact sequence of cohomologies associated to the short exact sequence of groups
$$
0\,\longrightarrow\, \mathbb{Z}/n\mathbb{Z}\,\longrightarrow\, \mathbb{G}_m
\,\stackrel{z\mapsto z^n}{\longrightarrow}\, \mathbb{G}_{m}\,\longrightarrow\, 0,
$$
it follows that $H^1_{{\acute{e}t}}(Y,\,\mathbb{Z}/n\mathbb{Z})$ is isomorphic to the 
$n$-torsion part
$$H^1_{{\acute{e}t}}(Y,\, \mathbb{G}_m)[n]\, .$$
Consequently,
using a generalization of Hilbert Theorem 90 (\cite[p.~124, Proposition~4.9]{Milne}),
it follows that
$$\op{Hom}(\pi_{1}^{\acute{e}t}(Y),\, \mathbb{Z}/n\mathbb{Z})
\,\simeq\, \op{Pic}(Y)[n]\, .$$
Now setting $Y\,=\,\op{M}_{\op{G}}^{\delta}$, where $\op{G}$ is connected semisimple 
affine algebraic group over $\mathbb{C}$, the following corollary of Theorem 
\ref{thm0} is obtained.

\begin{corollary}
The Picard group of $\op{M}_{\op{G}}^{\delta}$ is torsion-free. 
\end{corollary}

If $\op{G}$ is simply connected, the Picard group of $\op{M}_{\op{G}}$ is known to 
be $\mathbb{Z}$ \cite{KNR}. A result of \cite{BLS:94} says that the Picard group of
$\op{M}_{\op{G}}^{\delta}$ is torsion-free if $\op{G}$ is a classical semisimple group.

\section{Fundamental group of the moduli stack}

Let $\op{G}$ be a connected, reductive affine algebraic group defined over $\mathbb C$. Let $X$ be an irreducible 
smooth complex projective curve. The moduli stack of principal $\op{G}$-bundles on $X$ will be denoted by 
$\mathcal{M}_{\op{G}}$. It is well known that the stack $\mathcal{M}_{\op{G}}$ is algebraic \cite{LM}.

\subsection{Uniformization}Let $\op{G}$ be a connected, semi-simple affine algebraic group defined over $\mathbb 
C$. We now recall the uniformization theorem that
describes $\mathcal{M}_{\op{G}}$ as a quotient of the affine
Grassmannian \cite{BL}, \cite{F}, \cite{KNR}. Let $\op{LG}$ denote the loop group 
viewed as an ind--scheme over $\mathbb{C}$; we note that the set of
$\mathbb{C}$--points of $\op{LG}$ is just $\op{G}(\mathbb{C}((t)))$. The group of 
positive loops (respectively, the $\mathbb{C}$-valued points of the groups of positive loops) will be denoted by
$\op{L^{+}G}$ (respectively, $\op{G}(\mathbb{C}[[t]])$). The quotient
\begin{equation}\label{e1}
\mathcal{Q}_{\op{G}}\, :=\, \op{LG}/\op{L^{+}G}
\end{equation}
is the affine Grassmannian. The universal cover of $\op{G}$ will be denoted by $\widetilde{\op{G}}$.
The kernel of the projection map $\widetilde{\op{G}}\,\longrightarrow\, \op{G}$ 
is isomorphic to the fundamental group $\pi_1(\op{G})$.

Fix a point $p \,\in\, X$. Let
$\op{L_XG}$ denote the ind-sub group of $\op{LG}$ whose set of $\mathbb{C}$-valued points is $$\op{G}(\mathcal{O}_X(\ast p))
\,=\, \op{G}(\mathcal{O}_{X\setminus\{p\}})\, \subset\,
\op{G}(\mathbb{C}((t)))\, .$$ The first part of the following result is standard and can be found in (\cite{KNR}, \cite{BL}, \cite{F}),
while the second part is proved in \cite{DS}.

\begin{proposition}
There is a canonical isomorphism between the stacks
$\mathcal{M}_{\op{G}}$ and ${\op{L_X G}}\backslash \mathcal{Q}_{{\op{G}}}$. Moreover, the
quotient map $\mathcal{Q}_{{\op{G}}}\, \longrightarrow\, \mathcal{M}_{\op{G}}$ is locally
trivial in the \'etale topology.
\end{proposition}

We now recall some well known results on the objects described above; see Lemma 1.2 in \cite[p. 185]{BLS:94}.

\begin{proposition}[{\cite{BLS:94}}]\label{uniformization}
Let $X$ be an irreducible smooth complex projective curve and $\op{G}$ a connected semisimple complex 
affine algebraic group. Then the following four hold: 
\begin{enumerate}
\item $\pi_0(\op{LG})\,=\, \pi_1(\op{G})$. 

\item The quotient morphism $\op{LG}\,\longrightarrow\, \mathcal{Q}_{\op{G}}$ induces a
bijection $\pi_0(\op{LG})\,\longrightarrow\, \pi_0(\mathcal{Q}_{\op{G}})$. Each connected component of 
$\mathcal{Q}_{\op{G}}$ is isomorphic to $\mathcal{Q}_{\widetilde{\op{G}}}$ (defined
as in \eqref{e1} by substituting $\widetilde{\op{G}}$ in place of $\op{G}$). As before,
$\widetilde{\op{G}}$ denotes the simply connected cover of $\op{G}$. 

\item The group $\pi_0(\op{L_XG})$ is canonically isomorphic to 
$H^1(X,\,\pi_{1}({\op{G}}))$, i.e.
$$\pi_0(\op{L_XG})\cong H^1(X,\,\pi_{1}({\op{G}})).$$
Further via the universal coefficients theorem in cohomology, we get 
$$H^1(X;\,\pi_1(\op{G}))\,\cong\, \operatorname{Hom}(H_1(X,\,\mathbb{Z}),\,\pi_1(\op{G}))
\,=\,\operatorname{Hom}(\mathbb{Z}^{2g},\,\pi_1(\op{G}))\,\cong\, (\pi_1(\op{G}))^{2g}.$$

\item The group $\op{L_XG}$ is contained in the neutral component $(\op{LG})^0$ of 
$\op{LG}$.
\end{enumerate}
\end{proposition}

By Proposition \ref{uniformization} (cf.\ 
\cite[p. 186, Proposition~1.3]{BLS:94}), the set of connected components 
$\pi_0(\mathcal{M}_{\op{G}})$ has a canonical bijection
with the fundamental group $\pi_1(\op{G})$.

\begin{definition}\label{def:nontriv}
For any $\delta \,\in\, \pi_1(\op{G})$, let $\mathcal{M}^{\delta}_{\op{G}}$ denote the connected component of 
$\mathcal{M}_{\op{G}}$ corresponding to $\delta$. The component of $\op{LG}(\mathbb{C})$ corresponding to $\delta \in \pi_1(\op{G})$ will be denoted by $\op{LG}^{\delta}(\mathbb{C})$. 

 \end{definition}
Let $\zeta$ be any element in the component 
$\op{LG}^{\delta}(\mathbb{C})$.
By Proposition \ref{uniformization}(2), we get an action of 
$\zeta^{-1}\op{L_XG}\zeta$ on $\mathcal{Q}_{\widetilde{\op{G}}}$. We now recall the uniformization theorem for each 
component $\mathcal{M}_{\op{G}}^{\delta}$ \cite[Proposition 1.3(b)]{BLS:94}. The second statement of the following 
proposition is derived from \cite{DS}.

\begin{proposition}\label{twisteduniformization}For each $\delta \,\in\, \pi_1(\op{G})$, let $\zeta$ be any element in the
component $\op{LG}^{\delta}(\mathbb{C})$ (see Definition \ref{def:nontriv}).
There is a canonical isomorphism of stacks
$$\mathcal{M}^{\delta}_{\op{G}}\,\simeq\,(\zeta^{-1}\op{L_XG}\zeta)
\backslash \mathcal{Q}_{\widetilde{\op{G}}}\, .$$
Moreover the quotient map $\pi\,:\,\mathcal{Q}_{\widetilde{\op{G}}}\,\longrightarrow\,
\mathcal{M}^{\delta}_{\op{G}}$ is locally trivial in the \'etale topology. 
\end{proposition}

\subsection{Fundamental Groups }

The quotient $\mathbb C$--space $\mathcal{Q}_{\op{G}}$ in \eqref{e1} as constructed in 
the works of Beauville-Laszlo, Kumar and Laszlo-Sorger, \cite{BL}, \cite{K}, \cite{LS}, is an 
ind--scheme, which is a direct limit of a sequence of projective schemes. 
It turns out that when $\op{G}$ is simply connected, the ind--scheme $\mathcal{Q}_{\op{G}}$ is both reduced 
and irreducible, hence it is integral \cite[p.~508, Proposition 4.6]{LS}, \cite[p.~406--407, Lemma 6.3]{BL}. 
The affine Grassmannian $\mathcal{Q}_{\op{G}}=\op{LG}/\op{L}^{+}\op{G}$ can be realized as an 
inductive limit of reduced projective Schubert varieties \cite{K}, \cite{M}.

\begin{remark}
We do not need to assume that $\op{G}$ is semisimple for defining $\op{LG}$. 
The same definition works for any reductive group $\op{G}$.
\end{remark}

We now recall a lemma (Lemma \ref{newlemma}) whose proof can be found in Section 8 of \cite{PS} for 
$\op{G}=\op{GL}_n$. The general case follows from more general results in Section 4 of \cite{Na}. 
We also refer the reader to Theorem 1.6.1 and the paragraph after Theorem 1.6.1 in \cite{Zhu} for a 
more comprehensive discussion.

\begin{lemma}\label{newlemma}
The affine Grassmannian $\mathcal{Q}_{\op{G}}$ is homotopic to the based loop group $\Omega_{e}(K_G)$, where
$K_G$ is a compact form of $\op{G}$.
\end{lemma}

The following lemma is a direct consequence of Lemma \ref{newlemma}. 

\begin{lemma}\label{simplyconnected}
Assume that ${\op{G}}$ is semisimple and simply-connected.
Then $\pi_1(\mathcal{Q}_{{\op{G}}})$ is trivial.
\end{lemma}

\subsubsection{Topological Stacks}\label{fundamentalstack}

We refer the reader to papers of Behrang Noohi \cite{No1, No2, No3} for the notion of topological stacks and its associated homotopy 
theory. Topological stacks are defined in Section 13.2 in \cite{No1} and homotopy groups of topological stacks are discussed in 
Section 17 in \cite{No1}. We also refer the reader to Section 5.1 in \cite{No3} for more discussion of higher homotopy groups.

In \cite[Section 20]{No1}, the author constructs a functor that takes an algebraic stack over $\mathbb{C}$ to a 
topological stack (see Proposition 20.2 in \cite{No1}). Moreover this functor has nice properties --- it sends 
smooth morphisms to local fibrations and \'etale morphisms to local homeomorphisms. The stacks 
$\mathcal{M}_{\op{G}}^{\delta}$ are all algebraic (admitting locally finite presentation over $\mathbb{C}$) and 
hence in particular topological. Moreover, the affine Grassmannian $\mathcal{Q}_{\widetilde{\op{G}}}$ has a natural 
topology coming from its ind-variety structure. 
We refer the reader to Section \ref{sec:topological} for a proof of the following proposition. 

\begin{proposition}\label{prop:topological} 
The natural morphism $\pi\,:\, \mathcal{Q}_{\widetilde{\op{G}}}\,\longrightarrow\, \mathcal{M}_{\op{G}}^{\delta}$ 
in Proposition \ref{twisteduniformization} gives a morphism between the corresponding topological stacks. 
\end{proposition}

The long exact sequence in homotopy associated to a ``Serre fibration" of 
topological stacks can be found in Section 5.2 in \cite{No3}. We also refer the reader to Section 4.2 in 
\cite{No3} for discussions on quotient stacks. Throughout this paper, we consider the fundamental group of an 
algebraic stack to be the fundamental group of the associated topological stack. The following lemma is due to 
Behrang Noohi.

\begin{lemma}\label{topostackcomm}
Let $\mathfrak{X}$ be a filtered topological stack with filtration given by $\{\mathfrak{X}_i\}_{i\in \mathbb{N}}$
and $\mathfrak{X}\,=\,\cup_{i\in I} \mathfrak{X}_i$, then $\pi_1(\mathfrak{X})\,=\,\lim \pi_1(\mathfrak{X}_i)$.
\end{lemma}

\begin{proof}
We shall use the notion of the classifying space $f\,:\, X'\,\longrightarrow\, {\mathfrak X}$ for any 
topological stack \cite{No2}. This $X'$ is a topological space, and $f$ is a 
(representable) morphism with the property that the base extension $f_T$ of $f$, along any 
morphism $T\,\longrightarrow\, {\mathfrak X}$ with $T$ a topological space $T$, is a weak equivalence of 
topological spaces. We refer the reader to \cite{No2} for all these notions and the existence 
of such a topological space $X'$.
	
So we choose one classifying space, and let $\{X'_i\}_{i\in\mathbb{N}}$ be the filtration induced on $X'$ via 
pull back. Since each $X'_i\,\longrightarrow\, {\mathfrak X}_i$ is a weak equivalence, the result now reduces to the same 
statement of the lemma for topological spaces.
\end{proof}

\begin{theorem}\label{fundamentalgroup}
Assume the group $\op{G}$ to be a semisimple affine algebraic group but not necessarily simply connected.
For any $\delta \,\in\, \pi_1(\op{G})$, there is a natural isomorphism
$$\pi_1(\mathcal{M}_{G}^{\delta})\,\cong\, \pi_0(\op{L_XG})\,\cong\, H^1(X,\,\pi_{1}({\op{G}})).$$
\end{theorem}

\begin{proof}
Consider the quotient map $\pi$ in Proposition \ref{twisteduniformization}.
By the Proposition \ref{prop:topological}, this induces a map between the underlying topological stacks. 
Since this fibration is locally trivial with respect to
the \'etale topology, we have a long exact sequence of homotopy groups
\begin{equation}\label{e2}
\pi_1(\mathcal{Q}_{\widetilde{\op{G}}})\,\longrightarrow\,
\pi_1(\mathcal{M}_{\op{G}}^{\delta})\,
\stackrel{\eta}{\longrightarrow}\, \pi_0(\zeta^{-1} \op{L_XG} \zeta) 
\,\longrightarrow\, \pi_0(\mathcal{Q}_{\widetilde{\op{G}}} )\,\longrightarrow\, 0\, .
\end{equation}
associated to the Serre-fibration $\pi$ (see Theorem 5.2 in \cite{No3}). Now, from Lemma \ref{simplyconnected} it
follows that the homomorphism $\eta$ in \eqref{e2} is injective, and from
Proposition \ref{uniformization} we conclude that $\eta$ is surjective. Consequently, the homomorphism $\eta$
is an isomorphism.

Since $\op{L_XG}$ and $\zeta^{-1} \op{L_XG}\zeta$ are conjugate
(by $\zeta$), it follows that the two sets $\pi_0(\zeta^{-1} \op{L_XG} 
\zeta)$ and $\pi_0(\op{L_XG})$ are bijective. Now the theorem
follows from Proposition \ref{uniformization}.
\end{proof}

A consequence of Theorem \ref{fundamentalgroup} is the following corollary on the fundamental group of the moduli stacks
of principal bundles with a reductive group as a structure group.

\begin{corollary}\label{fundamentalreductivestack}
Let $\op{G}$ be a reductive complex affine algebraic group, and let $\mathcal{M}^{\delta}_{\op{G}}$ denote a component 
of the moduli stack of principal $\op{G}$--bundles on the smooth complex projective curve $X$,
where $\delta\, \in\, \pi_1(\op{G})$. Then the fundamental group $\pi_1(\mathcal{M}^{\delta}_{\op{G}})$ is a subgroup of the (abelian) group $H^1(X,\,\pi_1(\op{G/Z(G)}))\times 
H^1(X,\,\pi_1(\op{\op{G}/[\op{G},\,\op{G}]}))$ such that the quotient group $$\left(H^1(X,\,\pi_1(\op{G/Z(G)}))\times
H^1(X,\,\pi_1(\op{\op{G}/[\op{G},\,\op{G}]}))\right)/\pi_1(\mathcal{M}^{\delta}_{\op{G}})$$
is $H^1(X,\, \op{Z([G,\,G])})$, where $\op{Z([G,\,G])}$ is the center of 
$\op{[G,\,G]}$.
\end{corollary}

\begin{proof}
Let $\op{Z(G)}$ denote the center of $\op{G}$, and let $\op{[G,\,G]}$ be the commutator subgroup of $G$. Consider
the natural group homomorphism
$$f\,:\,\op{G}\,\longrightarrow\, (\op{G}/\op{Z(G)})\times (\op{G}/[\op{G},\,\op{G}])\, .$$ It is easy to see that the kernel
$\op{K}$ of $f$ is $\op{Z(G)}\bigcap [\op{G},\,\op{G}]$ which also coincides with the center $\op{Z([G,\,G])}$. Now,
since $\op{[G,\,G]}$ is semisimple, we conclude that $\op{K}\,=\, \text{kernel}(f)$ is a finite group. The corresponding morphism
of moduli stacks 
$$\mathcal{M}_f\,: \,\mathcal{M}_{\op{G}}^{\delta}\,\longrightarrow \,\mathcal{M}^{\delta_1}_{\op{G/Z(G)}}\times
\mathcal{M}^{\delta_2}_{\op{G}/\op{[G,\,G]}}$$ is an \'etale Galois cover with Galois group $H^1(X,\,\op{K})$; 
here $\mathcal{M}_{\rm G}^{\delta}$ denotes a particular component of the moduli stack
$\mathcal{M}_{\op{G}}$, 
while $\delta_1$ and $\delta_2$ are the images of $\delta$ in $\pi_1(\op{G/Z(G)})$ and
$\pi_1(\op{G}/\op{[G,\,G]})$ respectively under the quotient maps. 
Hence from the long exact sequence of homotopy groups associated to the above fibration $\mathcal{M}_f$ we see that 
$\pi_1(\mathcal{M}_{\op{G}}^{\delta})$ injects into $\pi_1( \mathcal{M}^{\delta_1}_{\op{G/Z(G)}})\times
\pi_{1}(\mathcal{M}^{\delta_1}_{\op{G}/[\op{G},\,\op{G}]})$ with quotient $H^1(X,\, \op{Z([G,\,G])})$.

Since $\op{G/Z(G)}$ is semisimple, Theorem \ref{fundamentalgroup} says that $$\pi_1( \mathcal{M}^{\delta_1}_{\op{G/Z(G)}})\,=\,
H^1(X,\,\pi_1(\op{G/Z(G)}))\,.$$ On the other hand, since ${\op{G}/[\op{G},\,\op{G}]}$ is a product of copies of
the multiplicative group ${\mathbb G}_m$, it follows
that $$\pi_{1}(\mathcal{M}^{\delta_1}_{\op{G}/\op[\op{G},\,\op{G}]})\, =\, H^1(X,\,\pi_1(\op{\op{G}/[\op{G},\,\op{G}]}))\, .$$
This completes the proof.
\end{proof}
The following consequence of Corollary \ref{fundamentalreductivestack} was observed by an anonymous referee and we thank him for his comment. 

\begin{corollary}
The rank (as an abelian group) of $\pi_1(\mathcal{M}_{\op{G}}^{\delta})$ is $2gd$, where $d
\,=\,\dim \op{Z({G})}$. In particular the fundamental groups of the moduli space $\op{M^{\delta}_G}$ (see Theorem \ref{thm0})
and that of the moduli stack $\mathcal{M}_{\op{G}}^{\delta}$ differ only on their torsion parts. 
\end{corollary}

\begin{proof}
The result follows from the following short exact sequence obtained from Corollary \ref{fundamentalreductivestack},
the additivity of rank in such sequences and the vanishing of the ranks of $H^1(X;\, \pi_1(\op{Ad(G)}))$ and
$H^1(X;\,\op{Z([G,\,G])})$. We have
$$0\,\longrightarrow\, \pi_1(\mathcal{M}_G^{\delta})\,\longrightarrow\, H^1(X;\, \pi_1(\op{Ad(G)}))\oplus
H^1(X;\, \pi_1(\mathbb{G}_m^d))\,\longrightarrow\, H^1(X;\,\op{Z([G,\,G])})\,\longrightarrow\, 0.$$ Here
$\op{Ad(\op{G})\,=\,\op{G}/Z(\op{G})}$ denotes the adjoint group of $\op{G}$. 
\end{proof}

\section{Fundamental groups}

In this section, we compute fundamental group of some twisted moduli stacks. We consider moduli stacks of certain 
reductive group $\op{C_A\widetilde{G}}$ associated to a central subgroup $\op{A}$ of $\op{\widetilde{G}}$. The idea 
to consider moduli stacks for these groups $\op{C_A\widetilde{G}}$ comes from the work of Beauville-Laszlo-Sorger 
\cite{BLS:94}. \subsection{Fundamental group of the twisted moduli stack}\label{sec:fundtwistedstack}

As before, 
let $\widetilde{\op{G}}$ be a semi-simple and simply connected affine complex
algebraic group. Given a subgroup $A$ of the center of $\widetilde{\op{G}}$, define
$$\op{G}\,:=\,\widetilde{\op{G}}/A\, .$$
Take any $\delta \,\in\, \pi_1(\op{G})$. We shall now recall from
\cite{BLS:94} the construction of a ``twisted" moduli stack 
$\mathcal{M}_{\widetilde{\op{G}}}^{\delta}$ dominating
$\mathcal{M}_{\op{G}}^{\delta}$.

For any positive integer $n$, the group of $n$-th roots of unity will be denoted
by $\mu_n$. We identify $\mu_n$ with $\mathbb{Z}/n\mathbb{Z}$ using the generator
$\exp(2\pi \sqrt{-1}/n)$ of $\mu_n$.
Fix an isomorphism
\begin{equation}\label{fe1}
A\,\simeq\, \prod_{j=1}^s\mu_{n_j}\, .
\end{equation}
Since $\prod_{j=1}^s\mu_{n_j}$ is canonically a subgroup of $T\,:=\,(\mathbb{G}_{m})^s$,
the isomorphism in \eqref{fe1} identifies $A$ with a subgroup of $T$. Next we
identify the quotient $\mathbb{G}_{m}/\mu_n$ with $\mathbb{G}_{m}$ via the endomorphism
$z\, \longmapsto\, z^n$ of $\mathbb{G}_{m}$. Using these, the quotient $T/A$ gets
identified with $T$. Let
\begin{equation}\label{f1}
C_{A}(\widetilde{\op{G}}) \,=\, (\widetilde{\op{G}}\times T)/A
\end{equation}
be the quotient by the diagonal subgroup $A$. The projection to the second factor
$$
C_{A}(\widetilde{\op{G}})\, \longrightarrow\, (\widetilde{\op{G}}/A)\times (T/A)
\,=\, \op{G}\times (T/A) \, \longrightarrow\, T/A\,=\, T
$$
induces a morphism of the moduli stacks
\begin{equation}\label{det}
\det \,:\, \mathcal{M}_{C_A(\widetilde{\op{G}})}\, \longrightarrow\, \mathcal{M}_{T}\, .
\end{equation}

Now, since $\widetilde{\op{G}}$ is simply connected, there is an isomorphism $$\rho\,:\, 
\pi_1({\op{G}})\,\cong\, A\, .$$ Take any $\vec{d}\,=\,(d_1,\, \dots,\, d_s)
\, \in\, {\mathbb Z}^s$ (see \eqref{fe1}) such that $0\,\leq\, d_i\,<\, n_i$ for all
$1\,\leq\, i\,\leq\, s$. We set
\begin{equation}\label{fe2}
\delta\, :=\, \rho^{-1}(\exp({2\pi} \sqrt{-1}(-{d_1}/{n_1})),\,\dots,\, \exp({2\pi} \sqrt{-1}(-{d_s}/{n_s})))\,. 
\end{equation} 
Let
\begin{equation}\label{stk-MG-delta}
\mathcal{M}_{\widetilde{\op{G}},A}^{\delta}\, :=\,
{\det}^{-1} ((\mathcal{O}_X(d_1p),\,\dots,\, \mathcal{O}_X(d_sp)))\,\subset\,
\mathcal{M}_{C_A{\widetilde{\op{G}}}}\, 
\end{equation}
be the sub-stack, where $\delta$ and 
$\vec{d}\,=\,(d_1,\,\dots,\,d_s)$ are related by \eqref{fe2}. Following 
\cite[Section 2]{BLS:94}, we shall call the stack $\mathcal{M}_{\widetilde{\op{G}},A}^{\delta}$ the 
{\em twisted moduli stack parametrizing $C_{A}(\widetilde{\op{G}})$--bundles with ``determinant" 
$(\mathcal{O}_X(d_1p),\,\dots,\, \mathcal{O}_X(d_sp))$}. 

It should be mentioned that the twisted principal $(C_{A}\widetilde{\op{G}})$--bundles, described above, 
can be realized as parahoric $\op{G}$--torsors on $X$ \cite{BS}, \cite{He}.
So $\mathcal{M}_{\widetilde{\op{G}},A}^{\delta}$ is also a moduli stack of parahoric $\op{G}$--torsors.

\begin{remark}
In \cite{BLS:94}, $\mathcal{M}_{\widetilde{\op{G}},A}^{\delta}$ is defined for arbitrary semi-simple groups (not 
necessarily simply connected) and is denoted by $\mathcal{M}_{\widetilde{\op{G}},A}^{d}$. The notation 
$\mathcal{M}_{\widetilde{\op{G}},A}^{\delta}$ is used in \cite{BLS:94} for an open and closed substack of 
$\mathcal{M}_{\widetilde{\op{G}},A}^{\delta}$. It was observed \cite{BLS:94} that for simply connected groups these 
two substacks coincide.
\end{remark}

The natural projection
$C_{A}(\widetilde{\op{G}})\, \longrightarrow\, \widetilde{\op{G}}/A \,=\, \op{G}$ induces a
surjective morphism of stacks
$$
\mathcal{M}_{\widetilde{\op{G}},A}^{\delta}\,\longrightarrow
\, \mathcal{M}_{\op{G}}^{\delta}\, .
$$
We now recall from \cite[Proposition 1.3 and Example 2.4]{BLS:94}, \cite{BS} and \cite{He} the uniformization theorem for twisted moduli stacks. 

\begin{proposition}\label{moretwisteduniformization}
Let $A$ denote a subgroup of the center $\op{Z}(\widetilde{\op{G}})$ of $\widetilde{\op{G}}$, and consider the group $\op{G}\,=\,\widetilde{\op{G}}/A$. Let $\zeta$
be any element of $\op{LG}^{\delta}(\mathbb{C})$. Then there is a canonical isomorphism 
$$\mathcal{M}_{\widetilde{\op{G}},A}^{\delta}\,\simeq\,
(\zeta^{-1}(\op{L_X}\widetilde{\op{G}})\zeta) 
\backslash \mathcal{Q}_{\widetilde{\op{G}}}\, ,$$
and moreover the natural fibration
$\pi\,:\, \mathcal{Q}_{\widetilde{\op{G}}} \,\longrightarrow\,
\mathcal{M}_{\widetilde{\op{G}},A}^{\delta}$ is locally trivial in the \'etale topology.
\end{proposition}

\begin{remark}
In the statement of Proposition \ref{moretwisteduniformization}, note that $\zeta$ is an element of 
$\op{LG}^{\delta}$. We explain the notation of conjugation by $\zeta$ in $\op{L}_X\widetilde{\op{G}}$. 
Consider the short exact sequence 
$$0\,\longrightarrow\, T/A\,\longrightarrow\, C_A(\widetilde{\op{G}})\,\longrightarrow\, 
\widetilde{\op{G}}/A\, \longrightarrow\, 0\, ,$$ 
where $\op{G}\,=\,\widetilde{\op{G}}/A$. Moreover $T/A$ is in the center of $C_A(\widetilde{\op{G}})$. 
Any two lifts of $\zeta$ to $\op{L}C_A(\widetilde{\op{G}})$ will differ by a central element. 
Consequently, conjugation in $\op{L}C_A({\widetilde{\op G}})$ by any lift of $\zeta$ is independent of the lift.
\end{remark}

\begin{proof}[Proof of Proposition \ref{moretwisteduniformization}]
We just sketch the main step to reduce to the untwisted case. First observe that $\widetilde{\op{G}}$ is the kernel 
of the natural homomorphism $C_{A}\widetilde{\op{G}}\,\longrightarrow\, T$. Now by construction, 
$$\mathcal{M}_{\widetilde{\op{G}}, A}^{\delta}\, =\, {\det}^{-1} ((\mathcal{O}_X(d_1p),\,\dots,\, 
\mathcal{O}_X(d_sp)))\, ,$$ where $\delta$ and $(d_1,\dots,d_s)$ are related by \eqref{fe2}. Observe that 
$(\mathcal{O}_X(d_1p),\,\dots,\, \mathcal{O}_X(d_sp))$ restricted to $X\backslash \{p\}$ is just 
$(\mathcal{O}_{X\backslash\{p\}},\,\dots,\, \mathcal{O}_{X\backslash\{p\}})$. Thus any principal 
$C_A(\widetilde{\op{G}})$--bundle with determinant $(\mathcal{O}_X(d_1p),\,\dots,\, \mathcal{O}_X(d_sp))$ 
restricted to the punctured curve $X\backslash \{p\}$, is a principal $\widetilde{\op{G}}$--bundle on $X\backslash 
\{p\}$. This construction is clearly functorial, in the sense that if a scheme $S$ parametrizes a family of 
$C_{A}(\widetilde{\op{G}})$--bundles on $X$ with determinant $(\mathcal{O}_X(d_1p),\,\dots,\, 
\mathcal{O}_X(d_sp))$, then the restriction of the family to $(X\backslash \{p\}) \times S$ gives a family of 
principal $\widetilde{\op{G}}$--bundles on $X\backslash\{p\}$ parametrized by $S$.

Now the proof follows as in the untwisted case by using \cite{DS} and the proof of Proposition 1.3 in \cite{BLS:94} 
(see also Remark 3.6 in \cite{BL}), but we outline the key steps for completeness. First consider the natural
homomorphism $C_{A}(\widetilde{\op{G}})\,\longrightarrow\, \op{T}$ 
which in turn gives a homomorphism of the corresponding loop groups $\op{det}\,:\,\op{L}{C_{A}(\widetilde{\op{G}})}\,
\longrightarrow\, \op{L}\op{T}$. Let 
us consider the ind-subscheme of $\op{L}{C_{A}(\widetilde{\op{G}})}$ given by 
$\op{L}{\widetilde{\op{G}}}^{\delta}\,:=\,\op{det}^{-1}(z^{-d_1},\dots,z^{-d_s})$. 
The discussion in the above paragraph and the uniformization theorem \cite{DS} together give the following isomorphism 
of stacks: $$\mathcal{M}_{\widetilde{\op{G}},A}^{\delta}\,\simeq \, 
\op{L}_X\widetilde{\op{G}}\backslash \op{L}{\widetilde{\op{G}}}^{\delta}/ \op{L}^{+}{\widetilde{\op{G}}}\, .$$ 
Let $\zeta$ be any element of $\op{L}\op{G}^{\delta}(\mathbb{C})$. Take any lift $\widetilde{\zeta}$ of $\zeta$ in 
$\op{L}\widetilde{\op{G}}^{\delta}$. Observe that multiplication by $\widetilde{\zeta}^{-1}$ gives an isomorphism of 
$\op{L}\widetilde{\op{G}}^{\delta}$ with $\op{L}\widetilde{\op{G}}$. Hence the result on uniformization follows. 
Local triviality follows directly from \cite{DS}.
\end{proof}

Recall that $\widetilde{\op{G}}$ is simply connected, and $\op{G}\, =\, \widetilde{\op{G}}/A$, where $A$ is a subgroup 
of the center of $\widetilde{\op{G}}$ isomorphic to $\pi_1(\op{G})$. Now as in Section \ref{fundamentalstack}, we apply 
the homotopy exact sequence to the above Serre-fibration $\pi$, to get the following:

\begin{corollary}\label{cor1}
For any $\delta \,\in\, \pi_1(\op{G})$, the above moduli stack 
$\mathcal{M}_{\widetilde{\op{G}},A}^{\delta}$ is simply connected.
\end{corollary}

\begin{proof}
Since $\widetilde{\op{G}}$ is simply connected, it follows from Proposition 
\ref{uniformization} that $\pi_0(\zeta^{-1}(\op{L_X}\widetilde{\op{G}})\zeta)$ is 
trivial. Hence the above mentioned homotopy exact sequence gives that
$$\pi_1(\mathcal{M}_{\widetilde{\op{G}},A}^{\delta})\,\simeq\, 
\pi_0(\zeta^{-1}(\op{L_X}\widetilde{\op{G}})\zeta)\,=\, \{1\}\, .$$ This completes the 
proof.
\end{proof}

\subsection{Notation}\label{stacknotation}

Let $\op{G}$ be a connected semisimple complex affine algebraic group, and let $\widetilde{\op{G}}$ be its 
universal cover. For a central subgroup $A$ of $\widetilde{\op{G}}$ isomorphic to $\pi_{1}(\op{G})$, henceforth we 
drop the subscript $A$ and denote by $\mathcal{M}_{\widetilde{\op{G}}}^{\delta}$ the twisted moduli stack 
$\mathcal{M}_{\widetilde{\op{G}},\, A}^{\delta}$.

\subsection{Fundamental group of the regularly stable locus}\label{se3.1}

Henceforth, we assume that $\text{genus}(X)\,=\, g\, \geq\, 2$.
Take an element $\delta$ of the center of a
simple and simply connected group $\widetilde{\op{G}}$.
If $g\, =\, 2$, then in this section, we assume that either $\op{G}\, \not=\, \text{SL}(2,{\mathbb C})$ or
$\delta\, \not=\, 1$.

We shall recall the definition of a regularly stable principal bundle \cite{BLS:94}, \cite{BH}; for this
we need the definition of a stable principal bundle which we also recall below \cite{Ra}.

\begin{definition} Let $H$ be a connected reductive affine algebraic group over $\mathbb{C}$. 
A principal $H$--bundle $E_H$ on $X$ is said to be {\it semistable} (respectively, {\it stable}) 
if for any given reduction $E_P \,\subset\, E_H$ of the structure group of $E_H$ to any proper 
parabolic subgroup $P \,\subsetneq\, H$ (not necessarily maximal), and any nontrivial dominant 
character $\chi \,:\, P \,\longrightarrow\, \mathbb{G}_m$ which is trivial on the center of $H$,
we have ${\rm degree}(\chi_*E_P) \,\leq\, 0$ (respectively, ${\rm degree}(\chi_*E_P) \,<\, 0$),
where $\chi_*E_P \,=\, E_P \times^{\chi} \mathbb{G}_a$ is the line bundle on $X$ associated to 
the principal $P$--bundle $E_P$ for the character $\chi$.
\end{definition}

It is known that a principal $H$--bundle $E_H$ is semistable (respectively, stable) if and only if for any maximal 
parabolic subgroup $P \,\subsetneq\, H$, and any section $s$ of the projection $E_H/P\, \longrightarrow\, X$, we 
have $\text{degree}(s^*T_{\rm rel}) \, \geq\, 0$ (respectively, $\text{degree}(s^*T_{\rm rel}) \, >\, 0$), where 
$T_{\rm rel}$ is the relative tangent bundle for the above projection $E_H/P\, \longrightarrow\, X$ \cite[Lemma 
2.1]{Ra}.

A principal
$\op{H}$--bundle $E$ on $X$ is called \textit{regularly stable} if
\begin{itemize}
\item $E$ is stable, and

\item the natural homomorphism from the center of $\op{H}$ to $\text{Aut}(E)$, given by the action 
of $\op{H}$ on $E$, is an isomorphism.
\end{itemize}

As before, $\mathcal{M}^{\delta}_{\widetilde{\op{G}}}$ denotes the twisted
moduli stack associated to the triple $(X,\, \op{G},\, \delta)$ (see Section \ref{stacknotation}). Let
$$
\mathcal{M}^{\delta,rs}_{\widetilde{\op{G}}}\, \subset\,
\mathcal{M}^{\delta}_{\widetilde{\op{G}}}
$$
be the open sub-stack defined by the regularly stable locus. Then there are the following natural inclusions
\begin{equation}
\label{eqn:inclusionofstacks}
\mathcal{M}^{\delta,rs}_{\widetilde{\op{G}}}\, \subset\,
\mathcal{M}^{\delta,s}_{\widetilde{\op{G}}}\,\subseteq\, \mathcal{M}^{\delta,ss}_{\widetilde{\op{G}}}\, \subset\,
\mathcal{M}^{\delta}_{\widetilde{\op{G}}}\, ,
\end{equation}
where the $\mathcal{M}^{\delta,s}_{\widetilde{\op{G}}}$ (respectively,
$\mathcal{M}^{\delta,ss}_{\widetilde{\op{G}}}$) denotes the open sub-stack of
$\mathcal{M}_{\widetilde{\op{G}}}^{\delta}$ given by the stable (respectively, semistable) locus.

\subsection{Presentation as quotient stacks}\label{sec:pres-quot-stack}

In this section, we recall following \cite{Ram1}, a presentation of 
$\mathcal{M}^{\delta,ss}_{\widetilde{\op{G}}}$ as a quotient stack. We closely follow and recall the 
constructions given in the proof of Lemma 7.3 in \cite{BLS:94} (see also Proposition 3.4 in \cite{KN}). 
As mentioned before, it is assumed that $g(X)\,\geq\, 2$.

Let $\op{G}$ be any semisimple group. Choose a faithful representation $\rho \,:\, \op{G} 
\,\longrightarrow \, \op{SL}_r$. For any principal $\op{G}$-bundle $P$, let 
$\rho_*(P) := P_{\rho}\,=\, P\times^\rho {\mathbb C}^r$ be the vector bundle associated 
to $P$ for the representation $\rho$.

Fix a closed point $p$ on the curve $X$. For any integer $n$ sufficiently large, we have
$H^1(X,\,P_{\rho}(np))\,=\,0$ for all all semistable principal $\op{G}$-bundles $P$. Indeed, this
follows from semicontinuity of cohomology and boundedness of semistable principal $\op{G}$-bundles.
Take $n$ to be sufficiently large. Set $m(n)\,=\,r(n+1-g)$, and consider the functor 
parametrizing locally free quotients $E$ of $\mathcal{O}_X^{\oplus m(n)}$ of rank $r$ and degree 
$rn$. This is clearly representable \cite{Ram1} by a scheme $R(n)$ along with a universal family 
$\mathcal{E}$. Moreover $R(n)$ is smooth for for all $n$ sufficiently large. By \cite[Sections 4.8. 
4.13.3]{Ram1}, we get a scheme $R_{\op{G}}(n)$ that represents the functor of global sections 
of the fiber bundle $\mathcal{E}/\op{G}$ on $X\times R(n)$ which is equivalent to the functor 
parametrizing principal $\op{G}$--bundles $P$ whose associated vector bundle $P_{\rho}(np)$ is a locally 
free quotient of $\mathcal{O}_X^{\oplus m(n)}$. By the discussion in the proof of Lemma 4.13.3
in \cite{Ram1}
we get that $R_{\op{G}}(n)$ is smooth for $n$ large enough and supports an universal 
family of principal $\op{G}$-bundles. Moreover the group $\Gamma_n\,=\,\op{GL}(m(n))$ acts on 
$R_{\op{G}}(n)$ and $R(n)$ and the morphism $R_{\op{G}}(n)\,\longrightarrow\, R(n)$ is $\Gamma_n$ 
equivariant.

Now assume as before that $\widetilde{\op{G}}$ is simply connected and $A\,=\, \prod_{j=1}^s \mu_{n_j}\subset T$ is a central subgroup of 
$\widetilde{\op{G}}$ such that $\widetilde{\op{G}}/A\,=\,\op{G}$. The group $C_A\widetilde{\op{G}}=(\widetilde{\op{G}}\times T)/A$ is reductive; we first embed 
$C_{A}\widetilde{\op{G}}$ into a reductive group $S\,=\,\prod_{i=1}^s \op{GL}_{N_i}\times T/A$ such that the center of 
$C_{A}\widetilde{\op{G}}$ goes to the center of $S$ (see the proof of Lemma 7.3 in \cite{BLS:94} for the construction of $S$). Now as before we have a map $\det:\mathcal{M}_{S}\rightarrow \mathcal{M}_{Z(S)}$, where $Z(S)$ is the center of $S$. For any element $d'=(d'_1,\dots, d'_{2s})\in \mathbb{Z}^{2s}$ and a closed point $p$, consider the element $(\mathcal{O}_X(d'_1p),\dots, \mathcal{O}_X(d'_{2s}p))$ of $\mathcal{M}_{Z(S)}$. We denote by $\mathcal{M}_S^{d'}$ the closed sub-stack $\det^{-1}(\mathcal{O}_X(d'_1p),\dots, \mathcal{O}_X(d'_{2s}p))$.
 In particular, we have the diagram 
\begin{equation}\label{eqn:arrows}
\xymatrix{
\mathcal{M}_{C_A{\widetilde{\op{G}}}}\ar[r] &\mathcal{M}_S\\
\mathcal{M}_{\widetilde{\op{G}}}^{\delta}\ar[u]\ar[r]&\mathcal{M}_{S}^{d'}\ar[u].
&}
\end{equation}
Here $\delta$ and $d'$ are related by the map between the centers of $C_A{\widetilde{\op{G}}}$ and $S$ and equation \eqref{fe2}.
Since Ramanathan's construction 
works for arbitrary reductive group, the above construction goes through with the role of ${\rm SL}_r$ being replaced by $S$. Thus 
we get a scheme $R_{C_A(\widetilde{\op{G}})}(n)$ along with a universal family of principal $C_A(\widetilde{\op{G}})$-bundles. The 
projection $C_A\operatorname{\widetilde{\op{G}}}\,\longrightarrow\, T/A\,\simeq\, T$ induces a map $$\det: 
R_{C_A\widetilde{\op{G}}}(n)\,\longrightarrow \,\mathcal{M}_T\, ,$$ where $T\,:=\,\mathbb{G}_m^{s}$. Fixing $\vec{d}$ and $\delta$
related by 
equation \eqref{fe2}, we define the scheme 
$R^{\delta}_{\widetilde{\op{G}}}(n)\,=\,\det^{-1}(\mathcal{O}_X(d_1p),\,\cdots,\,\mathcal{O}_X(d_sp))$. Similarly we also define the scheme $R^{d'}_{S}$. Since
$R_{C_A\widetilde{\op{G}}}(n)$ is smooth for $n$ large enough, and the morphism $\det$ is smooth,
this implies that $R_{\widetilde{\op{G}}}^{\delta}(n)$ is also
smooth.

For a faithful representation $\rho \,:\, C_A\widetilde{\op{G}} \,\longrightarrow\, S$, consider the stack 
$\mathcal{M}^{\delta}_{\widetilde{\op{G}}}(n)$ parametrizing $\delta$-twisted
$\widetilde{\op{G}}$-bundles $E$ such that the corresponding vector 
bundles $E_{\rho}(np)$ are generated by global sections and $H^1(X,\, E_{\rho}(np))\,=\,0$. 
By the above discussion, we get the following:

\begin{proposition} The stacks $\mathcal{M}^{\delta}_{\widetilde{\op{G}}}(n)$ can be presented as the quotient stack
$[R_{\widetilde{\op{G}}}^{\delta}(n)/\Gamma_n]$, where $R^{\delta}_{\widetilde{\op{G}}}(n)$ is a smooth scheme and $\Gamma_n$ is a 
reductive group. Moreover $R^{\delta}_{\widetilde{\op{G}}}(n)$ supports a family $\mathcal{W}$ of $\delta$-twisted $\widetilde{\op{G}}$--bundles along with a lift of the action of $\Gamma_n$.
\end{proposition}

\subsection{Proof of Proposition \ref{prop:topological}}\label{sec:topological}

We now observe that these spaces $R^{\delta}_{\widetilde{\op{G}}}(n)$ can be used to give the morphism 
of the underlying topological stacks 
$\pi^{\rm top} \,:\, \mathcal{Q}_{\widetilde{\op{G}}}\,\longrightarrow\, \mathcal{M}_{\op{G}}^{\delta}$ 
induced by the morphism of stacks 
$\pi : \mathcal{Q}_{\widetilde{\op{G}}}\,\longrightarrow\, \mathcal{M}_{\op{G}}^{\delta}$ 
in Proposition \ref{twisteduniformization}. 
Since $\mathcal{M}_{\op{G}}^{\delta}$ is a quotient of $\mathcal{M}_{\widetilde{\op{G}}}^{\delta}$ by the finite group 
$H^1(X,\,\pi_1(\op{G}))$, it is enough to consider the simply connected case. For each $n$, we define 
$X_n \,\subset\,\mathcal{Q}_{\widetilde{\op{G}}}$ by 
$$X_n\,=\,\{ {\bf g}\mathcal{P} \in \mathcal{Q}_{\widetilde{\op{G}}}\ :\ 
H^1(X,\, \rho_*(\pi({\bf g}))\otimes \mathcal{O}(np))\,=\, 0\}\, ,
$$
where $\mathcal{P}\, =\, {\op L}^+\widetilde{\op{G}}$.
It follows that $X_n\subset X_{n+1}$. The affine Grassmannian $\mathcal{Q}_{\widetilde{\op{G}}}$ has the structure of an 
ind-variety, and hence $X_n$ acquires a natural topology. By the proof of Lemma 3.2 in \cite{KN}, we see that each 
$X_n$ is open in $\mathcal{Q}_{\widetilde{\op{G}}}$
and that $\bigcup_{n\geq 0}X_n\,=\,\mathcal{Q}_{\widetilde{\op{G}}}$.
By definition
$$
X_n\,=\,\pi^{-1}(\mathcal{M}_{\widetilde{\op{G}}}^{\delta}(n))\,, 
$$
and it parametrizes a family of $\delta$-twisted $\widetilde{\op{G}}$--bundles. By the universality of 
$R_{\widetilde{\op{G}}}^{\delta}(n)$ (see Section 7.8 in \cite{KNR} and Section 3 in \cite{KN}), we get a family 
$\mathcal{F}_n$ of $\Gamma_n$--bundles on $X_n$ and $\Gamma_n$--equivariant morphism $\mathcal{F}_n\,\longrightarrow\, 
R^{\delta}_{\widetilde{G}}(n)$. Taking quotients, we get a morphism of their underlying topological quotient stacks 
$X_n\,\longrightarrow\, \mathcal{M}_{\widetilde{G}}^{\delta}(n)$. Taking the limit, we get the required morphism of 
topological stacks 
$\pi^{\op{top}} : \mathcal{Q}_{\widetilde{\op{G}}} \longrightarrow \mathcal{M}_{\widetilde{\op{G}}}^\delta$. 
This completes the proof of Proposition \ref{prop:topological}.

\subsection{Fundamental group of $\mathcal{M}^{\delta,rs}_{\widetilde{\op{G}}}$}

Let $R^{\delta,ss}_{\widetilde{\op{G}}}(n)$ (respectively, $R^{\delta,rs}_{\widetilde{\op{G}}}(n)$) be an open subscheme of 
$R^{\delta}_{\widetilde{\op{G}}}(n)$ such that the associated family of $\delta$ twisted principal $\widetilde{\op{G}}$-bundles
is semistable 
(respectively, regularly stable). Since our representation $\rho$ takes the center of $C_A\widetilde{\op{G}}$ to the center of $S$, it 
follows from \cite[Theorem 3.18]{RR} (see also \cite{Ram1}), that the canonical map 
$R_{C_A\widetilde{\op{G}}}(n)\,\longrightarrow\, R_{S}(n)$ preserves semistability.

For $n$ large enough, $\mathcal{M}_{\widetilde{\op{G}}}^{\delta,ss}\,\hookrightarrow\, \mathcal{M}_{\widetilde{\op{G}}}^{\delta}(n)$ 
and we get that $\mathcal{M}_{\widetilde{\op{G}}}^{\delta,ss}$ coincides with the quotient stack $[R^{\delta,ss}_{\widetilde{\op{G}}}(n)/\Gamma_n]$.

\begin{lemma}
For n large enough, the codimension of the complement of $R^{\delta,ss}_{\widetilde{\op{G}}}(n)$ in 
$R^{\delta}_{\widetilde{\op{G}}}(n)$ is at least two. In particular 
$\pi_1(\mathcal{M}_{\widetilde{\op{G}}}^{\delta,ss})=\pi_1(\mathcal{M}_{\widetilde{\op{G}}}^{\delta}(n))$.
\end{lemma}

\begin{proof}By \cite[Lemma 2.1]{BH} (see also \cite[Theorem II.6]{F}), the codimension of 
$\mathcal{M}_{\widetilde{\op{G}}}^{\delta,ss}$ in $\mathcal{M}_{\widetilde{\op{G}}}^{\delta}(n)$ is at least two. 
Since $R^{\delta}_{\widetilde{\op{G}}}(n)$ and $R^{\delta,ss}_{\widetilde{\op{G}}}(n)$ are $\Gamma_n$ torsors, 
this implies that the codimension of the complement of $R^{\delta,ss}_{\widetilde{\op{G}}}(n)$ in 
$R^{\delta}_{\widetilde{\op{G}}}(n)$ is at least two. Now by construction $R^{\delta}_{\widetilde{\op{G}}}(n)$ 
is smooth and hence $\pi_1(R^{\delta,ss}_{\widetilde{\op{G}}}(n))=\pi_1(R^{\delta}_{\widetilde{\op{G}}}(n))$. 
Thus the result follows.
\end{proof}

\begin{lemma}\label{lem1}
The fundamental group $\pi_1(\mathcal{M}^{\delta,rs}_{\widetilde{\op{G}}})$ is trivial if 
\begin{itemize}
\item either $g(X)\,\geq \,3$, or

\item $g\, =\, 2$ with either $\widetilde{\op{G}}\,\neq \,\op{SL}_2$ or $\delta \,\neq \,1$.
\end{itemize}
\end{lemma}

\begin{proof}
By \cite[Theorem 2.5]{BH2}, we get that the codimension of the complement of 
$\mathcal{M}^{\delta,rs}_{\widetilde{\op{G}}}$ in 
$\mathcal{M}^{\delta,ss}_{\widetilde{\op{G}}}$ is at least $2$ if $g(X)\geq 3$ and 
$\widetilde{\op{G}}\neq \op{SL}_2$ or $\delta \neq 1$ if $g(X)=2$. Since 
$R^{\delta,rs}_{\widetilde{\op{G}}}(n)$ (respectively 
$R^{\delta,ss}_{\widetilde{\op{G}}}(n)$) are both $\Gamma_n$--torsors, it implies that for $n$ 
large enough the codimension of the complement of $R^{\delta,rs}_{\widetilde{\op{G}}}(n)$ in 
$R^{\delta,ss}_{\widetilde{\op{G}}}(n)$ is at least two. Moreover, both 
$R^{\delta,rs}_{\widetilde{\op{G}}}(n)$ and $R^{\delta,ss}_{\widetilde{\op{G}}}(n)$ are 
smooth. Thus 
$\pi_1(R^{\delta,rs}_{\widetilde{\op{G}}}(n))\,=\,\pi_1(R^{\delta,ss}_{\widetilde{\op{G}}}(n))$. 
Thus for $n$ large enough, we get
$$\pi_1(\mathcal{M}_{\widetilde{\op{G}}}^{\delta,rs})\,=\,
\pi_1(\mathcal{M}_{\widetilde{\op{G}}}^{\delta,ss})
\,=\,\pi_1(\mathcal{M}_{\widetilde{\op{G}}}^{\delta}(n)).$$
Now by taking limits and applying Lemma \ref{topostackcomm}, we get that
$\pi_1(\mathcal{M}_{\widetilde{\op{G}}}^{\delta,rs})\,\simeq\,
\pi_1(\mathcal{M}^{\delta}_{\widetilde{\op{G}}})$.
\end{proof}

\subsection{Twisted Moduli spaces}\label{sec:twistMS}

The twisted moduli space $\text{M}^{\delta}_{\widetilde{\op{G}}}$ associated to the triple 
$(X,\, \op{G},\, \delta)$ is defined just as the twisted moduli stack is defined (see Equation \eqref{stk-MG-delta}). 
As before, let $A$ be the subgroup of the center of $\widetilde{\op{G}}$ isomorphic to $\pi_1(\op{G})$. 
\begin{definition}\label{def:twims}
The space $\text{M}^{\delta}_{\widetilde{\op{G}},A}$ is defined to be the moduli space 
of semistable principal $C_{A}\widetilde{\op{G}}$--bundles $E$ on $X$ (see  Equation \eqref{f1})
such that the associated principal $T$--bundle obtained by extending the structure group of $E$ 
using the homomorphism $\det$ in \eqref{det} is the principal $T$--bundle
corresponding to $(\mathcal{O}_X(d_1p),\,\dots,\, \mathcal{O}_X(d_sp))$, where $\delta$ and 
$\vec{d}\,=\,(d_1,\,\dots,\,d_s)$ are related by \eqref{fe2}. 

For notational simplicity, we drop the subscript $A$ and write $\text{M}^{\delta}_{\widetilde{\op{G}}}$ for 
$\text{M}^{\delta}_{\widetilde{\op{G}},\, A}$. 
\end{definition}

Let
$$
\text{M}^{\delta,rs}_{\widetilde{\op{G}}}\,\subset\,
\text{M}^{\delta}_{\widetilde{\op{G}}}
$$
be the twisted moduli space of regularly stable principal $\op{G}$--bundles associated 
to the triple $(X,\, \op{G},\, \delta)$.
As before, we assume that $g \,\geq\, 3$ and for $g \,=\, 2$, either $\op{G}
\,\neq\, \op{SL}_2(\mathbb{C})$ or $\delta \,\neq\, 1$. 

\subsubsection{Presentation of moduli spaces}\label{section:spaces}

We continue with the same notations as in Section \ref{sec:pres-quot-stack}. 
By our constructions in Section \ref{sec:pres-quot-stack}, we get a map $C_A{\widetilde{\op{G}}} \rightarrow S$ 
which preserves the center. This induces a morphism $\mathcal{M}_{C_A{\widetilde{\op{G}}}}\rightarrow \mathcal{M}_{S}$, 
which, in turn, give the morphisms $\mathcal{M}_{\widetilde{\op{G}}}^{\delta}\rightarrow \mathcal{M}_{S}^{d'}$ (see Diagram 
\ref{eqn:arrows}). By the discussion in Section \ref{sec:pres-quot-stack}, the stack 
$\mathcal{M}_{\widetilde{\op{G}}}^{\delta,ss}$ (respectively, $\mathcal{M}_{S}^{d',ss}$) is represented as a 
stack quotient of $R^{\delta,ss}_{\widetilde{\op{G}}}$ (respectively, $R^{d',ss}_{S}$) by a reductive group 
$\Gamma_n$. From classical theory of existence of good quotients of moduli spaces of vector bundles on a 
curve, it follows that $\text{M}_{S}^{d'}$ is a good quotient of $R^{d',ss}_{S}$ by $\Gamma_n$. Now since 
semistability is preserved (\cite[Theorem 3.18]{RR}), the construction of 
$\text{M}^{\delta}_{\widetilde{\op{G}}}$ as a good quotient of $R^{\delta,ss}_{\widetilde{\op{G}}}$ by 
$\Gamma_n$ follows from Lemma 5.1 in \cite{Ram1}. 

\begin{corollary}\label{cor2} The variety ${\rm 
M}^{\delta,rs}_{\widetilde{\op{G}}}$ is simply connected.
\end{corollary}

\begin{proof}
The coarse moduli space for $\mathcal{M}^{\delta,rs}_{\widetilde{\op{G}}}$
is ${\rm M}^{\delta,rs}_{\widetilde{\op{G}}}$.
The morphism to the coarse moduli space 
\begin{equation}\label{gerbe-rs}
\mathcal{M}^{\delta,rs}_{\widetilde{\op{G}}}\, \longrightarrow\, 
{\rm M}^{\delta,rs}_{\widetilde{\op{G}}}
\end{equation}
is a gerbe banded by the center $\op{Z}(\widetilde{\op{G}})$ of $\widetilde{\op{G}}$. 
A typical fiber over a point $x \in \rm M_{\widetilde{\op G}}^{\delta, rs}(\mathbb C)$ is given by the 
classifying stack $\Gamma_x\,:=\, B(\op{Z}(\widetilde{\op G}))$, whose associated topological stack is
connected. Moreover, banded gerbes are weak Serre fibrations \cite[Section 4.4]{No3}. Therefore, using
the homotopy exact sequence for the morphism \eqref{gerbe-rs}, 
we conclude that the homomorphism 
$$\pi_1(\mathcal M_{\widetilde{\op G}}^{\delta, rs})\,\longrightarrow\,\pi_1(\rm M_{\widetilde{\op G}}^{\delta, rs})$$ 
induced by the morphism \eqref{gerbe-rs} is surjective. 
Finally, $\pi_1(\mathcal{M}^{\delta,rs}_{\widetilde{\op{G}}})\,=\,1$ by Lemma \ref{lem1}.
\end{proof}

\section{Fundamental group of a moduli space}

\subsection{Simply connected simple groups}

Let $X$ be a compact connected Riemann surface of genus $g\, \geq\, 2$. 
Let $\op{G}$ be a simple group with simply connected cover $\widetilde{\op{G}}$. 
Consider $\pi_1(\op{G})$ as a subgroup $A$ of the center of $\widetilde{\op{G}}$. 
As before, for any $\delta \in \pi_1(\op{G})$, let ${\rm M}^{\delta}_{\widetilde{\op{G}}}={\rm M}^{\delta}_{\widetilde{\op{G}},A}$
be the twisted moduli space (see Definition \ref{def:twims}) of semistable bundles associated to $(X,\, \op{G},\, \delta)$. Recall that $\op{G}$ is isomorphic to $\widetilde{\op{G}}/A$.

\begin{proposition}\label{prop1}
The moduli space ${\rm M}^{\delta}_{\widetilde{\op{G}}}$ is simply connected.
\end{proposition}

\begin{proof}
First we consider the case where $g\,=\,2$, $\op{G}\, =\, \text{SL}(2,{\mathbb C})$ and $\delta\, =\, 1$. 
In this case, it follows from \cite[p.~27, Lemma 6.2 (ii) and p.~33, Theorem~2]{NR} that 
${\rm M}^{\delta}_{\widetilde{\op{G}}}\, =\, {\mathbb C}{\mathbb P}^3$, so this moduli space is 
simply connected.

Therefore, we assume that either $\op{G}\, \not=\, \text{SL}(2,{\mathbb C})$ or
$\delta\, \not=\, 1$ whenever $g\, =\, 2$. The Zariski open subset
\begin{equation}\label{e5}
{\rm M}^{\delta,rs}_{\widetilde{\op{G}}}\, \subset\, {\rm M}^{\delta}_{\widetilde{\op{G}}}
\end{equation}
is simply connected (Corollary \ref{cor2}). First observe that ${\rm M}^{\delta}_{\widetilde{\op{G}}}$ is a 
subspace of ${\rm M}^{\delta}_{C_A\widetilde{\op{G}}}$ realized as a fiber of the determinant map 
${\rm M}^{}_{C_A\widetilde{\op{G}}}\,\longrightarrow\, \op{M}_{T}$ in \eqref{det}. Since $C_A\widetilde{\op{G}}$ 
is reductive, we know by Corollary~3.4 of \cite{BH} that ${\rm M}^{\delta,rs}_{C_A\widetilde{\op{G}}}$ is the smooth 
locus of ${\rm M}^{\delta}_{C_A\widetilde{\op{G}}}$. Since ${\rm M}^{\delta,rs}_{C_A\widetilde{\op{G}}}$ is an \'etale 
locally trivial fiber bundle over a smooth variety with ${\rm M}^{\delta,rs}_{\widetilde{\op{G}}}$ as the typical fiber, 
it follows that ${\rm M}^{\delta,rs}_{\widetilde{\op{G}}}$ is the smooth locus of ${\rm M}^{\delta}_{\widetilde{\op{G}}}$. 

We note that if $Z$ is a normal projective variety, and $U_Z\, \subset\, Z$ is its smooth
locus, then the homomorphism $\pi_1(U_Z)\,\longrightarrow\, \pi_1(Z)$ induced
by the inclusion map is surjective. To prove this, take any desingularization 
$$
\sigma\ :\, \widehat{Z}\,\longrightarrow\, Z\, .
$$
By Zariski's main theorem (cf. \cite[p.~280, Ch.~III, Corollary~11.4]{Ha}) the fibers
of $\sigma$ are all connected. Therefore, the homomorphism
$$
\sigma_*\ :\, \pi_1(\widehat{Z})\,\longrightarrow\, \pi_1(Z)
$$
induced by $\sigma$ is surjective. Furthermore, the homomorphism
$\pi_1(\sigma^{-1}(U_Z))\,\longrightarrow\,\pi_1(\widehat{Z})$ induced
by the inclusion map is surjective, because $\widehat{Z}$ is smooth. But
$$
\sigma\vert_{\sigma^{-1}(U_Z)}\, :\, \sigma^{-1}(U_Z)\,\longrightarrow\, U_Z
$$
is an isomorphism. Hence combining the above observations we conclude that the
homomorphism $\pi_1(U_Z)\,\longrightarrow\, \pi_1(Z)$ is surjective.

Now, by construction (\cite[Lemma 7.3]{BLS:94} and \cite[Theorem~5.9]{Ram1}) the variety 
${\rm M}^{\delta}_{\widetilde{\op{G}}}$ is a good quotient (see Section \ref{section:spaces}) 
of a smooth scheme by a reductive group; hence it is normal. So the homomorphism of fundamental 
groups induced by the inclusion in \eqref{e5} is surjective. 
This implies that ${\rm M}^{\delta}_{\widetilde{\op{G}}}$ is simply connected, because
${\rm M}^{\delta,rs}_{\widetilde{\op{G}}}$ is simply connected by Corollary \ref{cor2}.
\end{proof}

\subsection{All simple groups}

As before, assume that $g\, \geq\, 2$.
Let $\op{G}$ be any simple group. Fix an element
\begin{equation}\label{de}
\delta\, \in\, \pi_1(\op{G})\, .
\end{equation}
As before, $\text{M}^\delta_{\op{G}}$ denotes the moduli space of semistable principal
$\op{G}$--bundles on $X$ of topological type $\delta$.

\begin{theorem}\label{thm1}
The moduli space ${\rm M}^\delta_{\op{G}}$ is simply connected.
\end{theorem}

\begin{proof}
Let $\gamma\, :\, \widetilde{\op{G}}\, \longrightarrow\, \op{G}$ be the universal
covering. The subgroup $\text{kernel}(\gamma)\,\subset\, \widetilde{\op{G}}$
will be denoted by $A$. This subgroup $A$ is contained in the center of
$\widetilde{\op{G}}$, and
\begin{equation}\label{f4}
A\,=\, \pi_1(\op{G})\, .
\end{equation}
Let
\begin{equation}\label{e6}
\Gamma\, :=\, \text{Hom}(\pi_1(X),\, A)\,=\, H^1(X,\, A)
\end{equation}
be the isomorphism classes of principal $A$--bundles on $X$. We note that
$\Gamma$ is a finite abelian group. The group structure on $A$ produces a
group structure on $\Gamma$ because $A$ is abelian.

Let ${\rm M}^{\delta}_{\widetilde{\op{G}}}$ be the twisted moduli space of semistable
principal bundles associated to $(X,\, \op{G},\, \delta)$, where $\delta$
is the element in \eqref{de}. We will construct an action
of $\Gamma$ on ${\rm M}^{\delta}_{\widetilde{\op{G}}}$. The homomorphism
$$
\widetilde{\op{G}}\times A \, \longrightarrow\, \widetilde{\op{G}}\, , \ \
(z,\, a)\, \longmapsto\, za
$$
produces a homomorphism
$$
\tau\, :\, C_{A}\widetilde{\op{G}}\times A\, \longrightarrow\, C_{A}\widetilde{\op{G}}\, ,
$$
where $C_{A}\widetilde{\op{G}}$ is the quotient group in \eqref{f1}. Given a principal
$C_{A}\widetilde{\op{G}}$--bundle $E$ and a principal $A$--bundle $F$ on $X$, we have
a principal $C_{A}\widetilde{\op{G}}$--bundle $\tau_*(E\times_X F)$, which is the
extension of structure group of the principal
$(C_{A}\widetilde{\op{G}}\times A)$--bundle $E\times_X F$, using the above homomorphism
$\tau$. Clearly, $\tau_*(E\times_X F)$ is semistable if and only if $E$ is semistable.
Consequently, we get an action on ${\rm M}^{\delta}_{\widetilde{\op{G}}}$ of the group
$\Gamma$ in \eqref{e6}
\begin{equation}\label{f3}
\widehat{\tau}\, :\, {\rm M}^{\delta}_{\widetilde{\op{G}}}\times\Gamma \, \longrightarrow\,
{\rm M}^{\delta}_{\widetilde{\op{G}}}\, .
\end{equation}

Consider the projection to the second factor
$$
C_{A}\widetilde{\op{G}}\,\longrightarrow\, (\widetilde{\op{G}}/A)\times (T/A)
\,=\, \op{G}\times (T/A) \, \longrightarrow\, \op{G}\, ,
$$
where $C_{A}\widetilde{\op{G}}$ is defined in \eqref{f1}. Given a principal
$C_{A}\widetilde{\op{G}}$--bundle on $X$, we have a principal $\op{G}$--bundle obtained by extending the
structure group using this homomorphism. This produces a morphism
${\rm M}^{\delta}_{\widetilde{\op{G}}}\, \longrightarrow\,\text{M}^\delta_{\op{G}}$. This
morphism clearly factors through the quotient ${\rm M}^{\delta}_{\widetilde{\op{G}}}/\Gamma$
for the above action of $\widehat{\tau}$ on ${\rm M}^{\delta}_{\widetilde{\op{G}}}$. The resulting morphism
$$
{\rm M}^{\delta}_{\widetilde{\op{G}}}/\Gamma\,\longrightarrow\, \text{M}^\delta_{\op{G}}
$$
is an isomorphism.

The homomorphism $\Gamma\, \longrightarrow\, \text{Aut}({\rm M}^{\delta}_{\widetilde{\op{G}}})$ given by the
above action of $\widehat{\tau}$ on ${\rm M}^{\delta}_{\widetilde{\op{G}}}$ is injective. To prove this, take any
nontrivial element
$h_0\, \in\, \text{Hom}(\pi_1(X),\, A)\,=\, \Gamma$. Let $F\, \longrightarrow\, X$ be the principal $A$--bundle
corresponding to $h_0$. Let
$$
h\, :\, \widetilde{X}\, \longrightarrow\, X
$$
be the \'etale Galois covering corresponding to $\text{kernel}(h_0)\, \subset\, \pi_1(X)$. The pullback
$h^*F\, \longrightarrow\,\widetilde{X}$ is a trivial principal $A$--bundle. Take any
principal $C_{A}\widetilde{\op{G}}$--bundle $E$ on $X$ such that
\begin{itemize}
\item the pullback $h^*E$ is regularly stable, and

\item $E$ lies in the moduli space ${\rm M}^{\delta}_{\widetilde{\op{G}}}$.
\end{itemize}
Since $h^*F$ is a trivial principal $A$--bundle, it follows that an isomorphism between
$E$ and $\tau_*(E\times_X F)$ produces an automorphism of $h^*E$; such an automorphism
of $h^*E$ is not given by
an element of the center of $C_{A}\widetilde{\op{G}}$ because $F$ is nontrivial. Since
$h^*E$ is regularly stable, it follows that the point of ${\rm M}^{\delta}_{\widetilde{\op{G}}}$
given by $E$ is not fixed by the action of
$h$ on ${\rm M}^{\delta}_{\widetilde{\op{G}}}$. Therefore, the above homomorphism
$$
\Gamma\, \longrightarrow\, \text{Aut}({\rm M}^{\delta}_{\widetilde{\op{G}}})
$$
is injective.

The fundamental group of the quotient of a path connected, simply connected, locally 
compact metric space by a faithful action of a finite group $B$ is the quotient of 
$B$ by the normal subgroup of it generated by all the isotropy subgroups 
\cite[p.~299, Theorem]{Am}. We shall apply this result to the action in \eqref{f3}. 
Note that the moduli space ${\rm M}^{\delta}_{\widetilde{\op{G}}}$ is simply 
connected by Proposition \ref{prop1}.

Since $A$ is abelian, the group $\Gamma$ in \eqref{e6} is generated by
the homomorphisms $\pi_1(X)\, \longrightarrow\, A$ such that the image is
a cyclic subgroup of $A$. Take any
\begin{equation}\label{th}
\theta\, :\, \pi_1(X)\, \longrightarrow\, A
\end{equation}
such that $\theta(\pi_1(X))$ is a cyclic subgroup of $A$; the order of $\theta$ will be denoted by
$m_0$. In view of the
above mentioned result of \cite{Am}, to prove that
$\text{M}^\delta_{\op{G}}$ is simply connected it suffices to show that the action
of $\theta$ on ${\rm M}^{\delta}_{\widetilde{\op{G}}}$ has a
fixed point. This result was proved in \cite[Lemma 7.4(b)]{BLS:94}. We give another proof of this fact and also recall the proof in \cite{BLS:94}. 
\subsubsection{First Proof}
It can be shown that there is a set of generators $\{a_1,\, \cdots, \, a_g,\,b_1,\,
\cdots ,\,b_g\}$ of standard type of $\pi_1(X)$ with a single relation
$$
\prod_{i=1}^g a_ib_ia^{-1}_ib^{-1}_i\, =\, 1
$$
such that
\begin{enumerate}
\item $\theta(b_i)\, =\, 1$ for all $1\, \leq\, i\, \leq\, g$, and

\item $\theta(a_i)\, =\, 1$ for all $2\, \leq\, i\, \leq\, g$.
\end{enumerate}
Indeed, for any element $\underline{\ell}\, :=\, (\ell_1, \,\dots, \, \ell_{2g})\,
\in\,({\mathbb Z}/m_0{\mathbb Z})^{2g}$ (recall that $m_0$ is the order of the image
of $\theta$), there is an element $A\, \in\, \text{Sp}(2g,
{\mathbb Z})$ such that
$$
\underline{\ell} A\,=\, (\ell_1,\, 0,\, 0, \,\dots, \, 0)\, .
$$
The image of the natural homomorphism from the mapping class group for $X$ to
the automorphism group $\text{Aut}(H_1(X,\, {\mathbb Z}))$ is the symplectic group
associated to the symplectic form on $H_1(X,\, {\mathbb Z})$
defined by the cap product. Combining these it follows that given any
standard presentation of $\pi_1(X)$, there is an element of the mapping class group
that takes it to a presentation of $\pi_1(X)$ satisfying the above conditions.
Clearly, the above presentation of $\pi_1(X)$ depends on $\theta$.

Fix a maximal compact subgroup
$$
\widetilde{\op{K}}\, \subset\, \widetilde{\op{G}}
$$
Let $F_{2g}$ denote the free group generated by $\{a_1,\, \dots, \, a_g,\,b_1,\,
\dots ,b_g\}$, so $\pi_1(X)$ is a quotient of
$F_{2g}$. Let ${\mathcal R}$ denote the space of all homomorphisms
$$\beta\, :\, F_{2g}\, \longrightarrow\, \widetilde{\op{K}}$$
such that $\beta(\prod_{i=1}^g a_ib_ia^{-1}_ib^{-1}_i) \,=\, \delta'$, where
$\delta'\, \in\, A$ is the element corresponding to $\delta\,\in\,
\pi_1(\op{G})$ (see \eqref{f4}, \eqref{de}). The group $\widetilde{\op{K}}$
acts on ${\mathcal R}$ through the conjugation action of $\widetilde{\op{K}}$ on
itself. A theorem of Ramanathan \cite{Ra} shows that
$$
{\rm M}^{\delta}_{\widetilde{\op{G}}}\,=\, {\mathcal R}/\widetilde{\op{K}}.
$$ The action $\widehat{\tau}$ (see \eqref{f3})
of $\theta$ (see \eqref{th}) on ${\rm 
M}^{\delta}_{\widetilde{\op{G}}}$ sends any homomorphism $\beta$ as above to the
homomorphism defined as follows:
\begin{itemize}
\item $b_i\, \longmapsto\, \beta(b_i)$ for all $1\, \leq\, i\, \leq\, g$,

\item $a_i\, \longmapsto\, \beta(a_i)$ for all $2\, \leq\, i\, \leq\, g$, and

\item $a_1\, \longmapsto\, \beta(a_1)\theta(a_1)$.
\end{itemize}

Define the subset of $\widetilde{\rm K}^{^3}$
$$
\widetilde{\mathcal S}\,:=\,\{(x_1,\, x_2,\, x_3)\, \in\, \widetilde{\rm K}^{^3}\, \mid\,
[x_1,\, x_2]\,=\, \delta',\, [x_3,\, x_1]\,=\, \theta(a_1),\, [x_3,\, x_2]\,=\,1\}\, .
$$
The number of connected components of $\widetilde{\mathcal S}$ coincides with that
of the quotient space $${\mathcal S}\,:=\, \widetilde{\mathcal S}/\widetilde{\op{K}}$$
because $\widetilde{\op{K}}$ is connected. The set of
connected components of ${\mathcal S}$ is described in \cite[p.~6, Theorem~1.5.1(3)]{BFM}; if we set
$$
C\, =\,
\begin{pmatrix}
1 & \delta' & \theta(a_1)^{-1}\\
(\delta')^{-1} & 1 & 1\\
\theta(a_1) & 1 & 1
\end{pmatrix}
$$
in \cite[p.~6, Theorem~1.5.1]{BFM}, and $\op{G}$ in \cite[p.~6, Theorem~1.5.1]{BFM}
to be $\widetilde K$, then the above quotient ${\mathcal S}$ coincides
with the space ${\mathcal T}_{\op{G}}(C)$ in \cite[p.~6, Theorem~1.5.1]{BFM}.
Setting $k\,=\,1$ in \cite[p.~6, Theorem~1.5.1(3)]{BFM} we
conclude that ${\mathcal T}_{\op{G}}(C)\,=\, {\mathcal S}$ is nonempty because the
Euler $\varphi$--function sends $1$ to $1$.

Take any triple $(x_1,\, x_2,\, x_3)\, \in\, \widetilde{\mathcal S}$. Let
$$\beta_0\,\in \,\text{Hom}(F_{2g},\, \widetilde{\op{K}})$$ be the homomorphism
defined by
\begin{itemize}
\item $b_i\, \longmapsto\, 1$ for all $2\, \leq\, i\, \leq\, g$,

\item $a_i\, \longmapsto\, 1$ for all $2\, \leq\, i\, \leq\, g$,

\item $a_1\, \longmapsto\, x_1$, and

\item $b_1\, \longmapsto\, x_2$.
\end{itemize}
Note that $\beta_0\, \in\, {\mathcal R}$. Let
$$
\beta'_0\,\in\, {\mathcal R}/\widetilde{\op{K}}\,=\, 
{\rm M}^{\delta}_{\widetilde{\op{G}}}
$$
be the image of $\beta_0$ under the quotient map. We have
$x_3x_1x^{-1}_3\,=\, x_1\theta(a_1)$, because $[x_3,\, x_1]\,=\, \theta(a_1)$. In view
of this and the third condition that $[x_3,\, x_2]\,=\,1$, we conclude from the above
description of the action of $\Gamma$ on ${\mathcal R}/\widetilde{\op{K}}\,=\,
{\rm M}^{\delta}_{\widetilde{\op{G}}}$ that the above point $\beta'_0$ is fixed
by the action of $\theta$. As noted before, this completes the proof using
\cite[p.~299, Theorem]{Am}.
\end{proof}

\subsubsection{Second Proof}
The following proof is well known \cite[Lemma 7.2(b)]{BLS:94}, 
but we recall it for completeness of the exposition:

Recall that $\Gamma\,=\,H^1(X,\,A)$ acts on $\op{M}_{\widetilde{\op{G}}}^{\delta}$, where 
$\widetilde{\op{G}}$ is simply connected. Thus every element of $\gamma \,\in\, \Gamma$ 
gives an automorphism of $\op{M}_{\widetilde{\op{G}}}^{\delta}$ of finite order. First following an argument in
\cite[Corollary 6.3]{KNR}, we show that $\op{M}_{\widetilde{\op{G}}}^{\delta}$ is unirational. 
By uniformization theorem, Proposition \ref{uniformization}, we get a surjection from the affine Grassmannian 
$\mathcal{Q}_{\widetilde{\op{G}}}$ to the moduli stack $\mathcal{M}_{\widetilde{\op{G}}}^{\delta}$. 
In particular there is a surjection from an open subset of $\mathcal{Q}_{\widetilde{\op{G}}}$ parametrizing 
semistable bundle to $\op{M}_{\widetilde{\op{G}}}^{\delta}$. Since $\mathcal{Q}_{\widetilde{\op{G}}}$ is a direct 
limit of an increasing sequence of generalized Schubert varieties, it follows that 
$\op{M}_{\widetilde{\op{G}}}^{\delta}$ is unirational. 

\begin{lemma}
Let $Y$ be an unirational, projective variety over $\mathbb{C}$. Then any finite 
order automorphism of $Y$ must have a fixed point. 
\end{lemma}

\begin{proof}
Let us assume that $Y$ is smooth. Since $H^i(Y,\,\mathcal{O}_Y)\,=\,0$ for all 
$i\,>\,0$, by the holomorphic Lefschetz fixed-point formula, any finite order 
automorphism of $Y$ must have a fixed point. Thus we are done. If $Y$ is singular, 
let $C$ be the cyclic group generated by the finite order automorphism. Let 
$\widetilde{Y}$ be a $C$-equivariant resolution of singularities, \cite{BM}, \cite{V}, of 
$\widetilde{Y}$. By the previous step, we get a fixed point of $\widetilde{Y}$ under 
the action of any element $c\,\in\, C$. Since the resolution is $C$-equivariant, we 
get a fixed point of $Y$ under the action of $c$. This completes the proof.
\end{proof}

\subsection{The case of reductive groups}\label{addition}

First assume that $\op{G}$ is any connected semisimple affine algebraic group defined
over $\mathbb C$. Take any $\delta\, \in\, \pi_1(\op{G})$. Let ${\rm 
M}^{\delta}_{\op{G}}$ denote the moduli space of semistable principal 
$\op{G}$--bundles on $X$ of topological type $\delta$.

\begin{corollary}\label{cor3}
The moduli space ${\rm M}^{\delta}_{\op{G}}$ is simply connected.
\end{corollary}

\begin{proof}
Let $\op{Z}\, \subset\, \op{G}$ be the center. The quotient $\op{G}/\op{Z}$ is 
isomorphic to $\prod_{i=1}^d \op{G}_i$, where each $\op{G}_i$ is simple with trivial 
center. The image of $\delta$ in $\pi_1(\op{G}_i)$ under the quotient map $$\op{G}/\op{Z}\, 
=\, \prod_{i=1}^d \op{G}_i\,\longrightarrow\, \op{G}_i$$ will be denoted by
$\delta_i$. Let ${\rm M}^{\delta_i}_{\op{G}_i}$ be the moduli space of semistable principal 
$\op{G}_i$--bundles on $X$ of topological type $\delta_i$.

The isomorphism classes of principal $\op{Z}$--bundles on $X$ will be denoted by
$\Gamma$. The homomorphism $\op{G}\times \op{Z}\, \longrightarrow\, \op{G}$,
$(x,\, z) \, \longmapsto\, xz$, produces an action of the abelian group $\Gamma$
on ${\rm M}^{\delta}_{\op{G}}$. We have
\begin{equation}\label{y1}
\prod_{i=1}^d{\rm M}^{\delta_i}_{\op{G}_i}\,=\,
{\rm M}^{\delta}_{\op{G}}/\Gamma\, .
\end{equation}
Now $\prod_{i=1}^d{\rm M}^{\delta_i}_{\op{G}_i}$ is simply
connected by Theorem \ref{thm1}. Therefore, from \eqref{y1} we conclude that
${\rm M}^{\delta}_{\op{G}}$ is simply connected.
\end{proof}

Finally, let $\op{G}$ be any connected reductive affine algebraic group defined over
$\mathbb C$. The commutator subgroup $[\op{G},\, \op{G}]$ is connected semisimple, and
there is a short exact sequence of groups
\begin{equation}\label{y2}
1\, \longrightarrow\, [\op{G},\, \op{G}]\, \longrightarrow\,\op{G}
\, \stackrel{q}{\longrightarrow}\, \op{Q}\,:=\, \op{G}/[\op{G},\, \op{G}]\, \longrightarrow\,
1\, ,
\end{equation}
where the quotient $\op{Q}$ is a product of copies of the multiplicative group
${\mathbb G}_m$. 

Take any $\delta\, \in\, \pi_1(\op{G})$. The image of $\delta$ in
$\pi_1(\op{Q})$ under the above projection $q$ will be denoted by $\alpha$. Let ${\rm
M}^{\delta}_{\op{G}}$ denote the moduli space of semistable principal
$\op{G}$--bundles on $X$ of topological type $\delta$. The moduli space of
principal $\op{Q}$--bundles on $X$ of topological type $\alpha$ will be denoted
by $J^{^\alpha}_{\op{Q}}(X)$. We note that $J^{^\alpha}_{\op{Q}}(X)$ is isomorphic to
$(\text{Pic}^0(X))^d$, where $d$ is the dimension of $\op{Q}$. Therefore, we have
$$
\pi_1(J^{\alpha}_{\op{Q}}(X))\, =\, {\mathbb Z}^{2gd}\, .
$$

The projection $q$ in \eqref{y2} induces a morphism
\begin{equation}\label{y3}
\widetilde{q}\, :\, {\rm M}^{\delta}_{\op{G}}\,\longrightarrow\,
J^{\alpha}_{\op{Q}}(X)\, .
\end{equation}

\begin{corollary}\label{cor4}
The homomorphism
$$
\widetilde{q}_*\, :\, \pi_1({\rm M}^{\delta}_{\op{G}})\,\longrightarrow\,
\pi_1(J^{\alpha}_{\op{Q}}(X))
$$
induced by the projection $\widetilde{q}$ in \eqref{y3} is an isomorphism.
\end{corollary}

\begin{proof}
Let $\op{Z}$ denote the center of $[\op{G},\, \op{G}]$. The moduli space
of principal $\op{Z}$--bundles on $X$ will be denoted by $\Gamma$.

The projection $\widetilde{q}$ in \eqref{y3} is surjective. It can
be shown that $\widetilde{q}$ is \'etale locally trivial as follows. For that, let
$\op{Z}_0\, \subset\, \op{Z}$ be the connected component containing the
identity element. Let
$$
q'\, :\, \op{G}\, \longrightarrow\, \op{G}/\op{Z}_0
$$
be the natural projection. Let
$$
\alpha'\, :=\, q'_*(\delta) \, \in\, \pi_1(\op{G}/\op{Z}_0)
$$
be the image under the homomorphism $q'_*\, :\, \pi_1(\op{G}) \, \longrightarrow\, \pi_1(\op{G}/\op{Z}_0)$
induced by $q'$. Let ${\rm M}^{\alpha'}_{\op{G}/\op{Z}_0}$ be the corresponding moduli space
of semistable principal $\op{G}/\op{Z}_0$--bundles on $X$. Let
$$
\varpi\, :\, {\rm M}^{\delta}_{\op{G}} \, \longrightarrow\, {\rm M}^{\alpha'}_{\op{G}/\op{Z}_0}
\times J^{\alpha}_{\op{Q}}(X)
$$
be the morphism of moduli spaces corresponding to the surjective homomorphism
$$
\op{G}\, \longrightarrow\, (\op{G}/\op{Z}_0)\times (\op{G}/[\op{G},\, \op{G}])\, ,\ \
z\, \longmapsto\, (q'(z),\, q(z))\, .
$$
It is straightforward to check that
\begin{equation}\label{et}
\widetilde{q}\,=\, p_2\circ \varpi\, ,
\end{equation}
where $p_2\, :\, {\rm M}^{\alpha'}_{\op{G}/\op{Z}_0}
\times J^{\alpha}_{\op{Q}}(X)\, \longrightarrow\, J^{\alpha}_{\op{Q}}(X)$ is the natural
projection, and $\widetilde{q}$ is the map in \eqref{y3}.

Consider the finite abelian group $\op{Z}_1\,:=\, [\op{G},\, \op{G}]\bigcap \op{Z}_0\, \subset\, \op{G}$. Let
${\rm M}_{\op{Z}_1}$ be the group of principal $\op{Z}_1$--bundles on $X$. The group ${\rm M}_{\op{Z}_1}$
acts on ${\rm M}^{\delta}_{\op{G}}$, and this action of ${\rm M}_{\op{Z}_1}$ on ${\rm M}^{\delta}_{\op{G}}$
takes any fiber of $\widetilde{q}$ to itself. In fact, we have
$$
\widetilde{q}^{-1}(t)/{\rm M}_{\op{Z}_1}\,=\, {\rm M}^{\alpha'}_{\op{G}/\op{Z}_0}\, .
$$
Therefore, from \eqref{et} it is deduced that $\widetilde{q}$ is \'etale locally trivial.

Hence by Corollary \ref{cor3}, we get that $\op{M}$ is simply connected. Now from the long exact
sequence of homotopy groups associated to the fiber bundle in \eqref{y3} we conclude
that the homomorphism $\widetilde{q}_*$ is an isomorphism.
\end{proof}

\begin{remark}\label{remf}
Take $\mathcal G$ to be any connected complex affine algebraic group.
Let $\op{G}$ be the quotient of $\mathcal G$ by the unipotent radical of $\mathcal G$, so $\op{G}$ is a
connected complex reductive affine algebraic group. For any $\delta\, \in\, \pi_1({\mathcal G})\,=\,
\pi_1(\op{G})$, the natural projection 
$$
{\rm M}^{\delta}_{\mathcal G}\, \longrightarrow\, {\rm M}^{\delta}_{\op{G}}
$$
is surjective with contractible fibers, in particular this map
${\rm M}^{\delta}_{\mathcal G}\, \longrightarrow\, {\rm M}^{\delta}_{\op{G}}$
induces an isomorphism of fundamental groups. Consequently, Theorem
\ref{thm0} computes the fundamental group of ${\rm M}^{\delta}_{\mathcal G}$.
\end{remark}

\section*{Acknowledgements}

We thank the two referees for very helpful comments. We are very grateful to Behrang Noohi for Lemma 
\ref{topostackcomm}. We thank S. Lawton and D. Ramras for making us aware of the question addressed here. We thank 
the Institute for Mathematical Sciences in the National University of Singapore for hospitality while this work was 
being completed. The first author is supported by a J. C. Bose Fellowship. The second author was supported in part 
by a Simons Travel Grant and by NSF grant DMS-1361159 (PI: Patrick Brosnan) and also by the Science and Engineering 
Research Board, India (SRG/2019/000513). The first and second author were also supported by the Department of 
Atomic Energy, India, under project no. 12-R\&D-TFR-5.01-0500.

\end{document}